\documentclass[12pt,reqno]{amsart}
\usepackage{fullpage}
\usepackage{times}
\usepackage[utf8]{inputenc}
\usepackage[english]{babel}
\usepackage[mathscr]{euscript}
\usepackage{amsmath,amssymb,amsthm, graphicx, enumitem, xcolor}
\allowdisplaybreaks
\usepackage[colorlinks=true, linkcolor=blue, citecolor=blue, urlcolor=blue]{hyperref}
\usepackage{comment}

\AtBeginDocument{%
	\def\MR#1{}
}
 
\title{On an Erd\H{o}s similarity problem in the large}

\author{Xiang Gao}
\address{Department of Mathematics, Hubei Key Laboratory of Applied Mathematics, Hubei University, Wuhan 430062, China}
\email{gaojiaou@gmail.com}

\author{Yuveshen Mooroogen}
\address{Department of Mathematics, University of British Columbia, Vancouver V6T 1Z2, Canada}
\email{yuveshenm@math.ubc.ca}

\author{Chi Hoi Yip}
\address{School of Mathematics\\ Georgia Institute of Technology\\ Atlanta, GA 30332\\ United States}
\email{cyip30@gatech.edu}


\newtheorem{thm}{Theorem}[section]
\newtheorem{lem}[thm]{Lemma}
\newtheorem{prop}[thm]{Proposition}
\newtheorem{cor}[thm]{Corollary}
\newtheorem{question}[thm]{Question}

\newtheorem{externaltheorem}{Theorem}

\theoremstyle{definition}

\newtheorem{ex}[thm]{Example}
\newtheorem{rem}[thm]{Remark}

\newcommand{\Q}{\mathbb{Q}}
\newcommand{\R}{\mathbb{R}} 
\newcommand{\N}{\mathbb{N}}
\newcommand{\Z}{\mathbb{Z}}
\renewcommand{\P}{\mathbb{P}}
\newcommand{\fract}[1]{\left \langle #1 \right \rangle}

\begin{document}

\begin{abstract}
In a recent paper, Kolountzakis and Papageorgiou ask if for every $\epsilon \in (0,1]$, there exists a set $S \subseteq \mathbb{R}$ such that $\vert S \cap I\vert \geq 1 - \epsilon$ for every interval $I \subset \mathbb{R}$ with unit length, but that does not contain any affine copy of a given increasing sequence of exponential growth or faster. This question is an analogue of the well-known Erd\H{o}s similarity problem. In this paper, we show that for each sequence of real numbers whose integer parts form a set of positive upper Banach density, one can explicitly construct such a set $S$ that contains no affine copy of that sequence. Since there exist sequences of arbitrarily rapid growth that satisfy this condition, our result answers Kolountzakis and Papageorgiou's question in the affirmative. A key ingredient of our proof is a generalization of results by Amice, Kahane, and Haight from metric number theory. In addition, we construct a set $S$ with the required property---but with $\epsilon \in (1/2, 1]$---that contains no affine copy of $\{2^n\}$.
\end{abstract}

\maketitle

\section{Introduction}

Many results in number theory, combinatorics, harmonic analysis, and geometric measure theory show that if a set is ``large'' in some quantitative sense, relative to some ambient space $\R^n$ or $\Z^n$, then it must contain a ``copy'' of a prescribed configuration (e.g. a given pattern, or the vertices of a polyhedron). One famous example is Szemerédi’s theorem \cite{Szemeredi75}, which asserts that if a set of integers has positive upper Banach density, then it must contain arbitrarily long arithmetic progressions.

The well-known \textit{Erd\H{o}s similarity problem} also belongs to this family of problems. To state it, we first recall the following terminology from \cite{Kolountzakis97, LabaPramanik09, CLP23}: a set $A \subseteq \R$ is \textit{universal} in a class $\mathscr{S}$ of subsets of $\R$ if every set $S \in \mathscr{S}$ contains a nontrivial affine copy of $A$; that is, if there exist $t \in \R$ and $x \in \R \setminus \{0\}$ such that $xA + t \subseteq S$. Let $\mathscr{E}$ be the collection of Lebesgue measurable subsets of $\mathbb{R}$ that have positive Lebesgue measure. One can show using the Lebesgue density theorem that every finite set $A \subset \mathbb{R}$ is universal in $\mathscr{E}$. Motivated by this observation, Erd\H{o}s \cite{Erdos74} conjectured that all infinite sets are non-universal in $\mathscr{E}$. This conjecture is known as the Erd\H{o}s similarity problem. It suffices to prove it for all sequences of real numbers decreasing to $0$ \cite{S00}. 

\begin{question}\label{q:erdos}
    Does there exist a decreasing sequence $a_n \to 0$ that is universal in $\mathscr{E}$?
\end{question}

This question is currently open, but several classes of non-universal sequences are known \cite{Falconer84, Bourgain87, Kolountzakis97, HumLac98, GLW23}. We refer the reader to \cite{S00} for a comprehensive survey of this problem. Intuitively, the sparser the sequence $\{a_n\}$ is, the easier it should be for some of its affine copies to be contained in a set of positive measure. It is known that if $\{a_n\}$ is \textit{sublacunary} (that is, if $\lim_{n \to \infty} a_{n+1}/a_n=1$), then it is not universal in $\mathscr{E}$ \cite{Falconer84, Eigen85}. This decay condition ensures that $\{a_n\}$ is not \textit{too} sparse. Kolountzakis \cite{Kolountzakis97} showed that the condition that $\{a_n\}$ is sublacunary can be weakened to the assumption that the sequence $\{a_n\}$ contains a large subset with ``large gaps".  On the other hand, Bourgain \cite{Bourgain87} showed that if $A$ has special arithmetic structure in the sense that $A$ can be written as the sumset of three infinite sets, then it is not universal in $\mathscr{E}$. To the best of our knowledge, Bourgain's result is the only known theorem that gives examples of non-universal sequences $\{a_n\}$ such that $-\log a_n \neq o(n)$. For exponential sequences that do not enjoy this special arithmetic structure---for example $\{2^{-n}\}$---the problem remains open \cite[page 2]{K23}. 

In this paper, we consider the following analogue of the Erd\H{o}s similarity problem. For each $\epsilon \in (0,1]$, consider the collection $\mathscr{R}(\epsilon)$ of measurable sets $S \subseteq \mathbb{R}$ such that $\vert S \cap I\vert \geq 1 - \epsilon$ for every interval $I \subset \mathbb{R}$ with unit length. We refer to the elements of $\mathscr{R}(\epsilon)$ as $(1-\epsilon)$-\textit{large sets}. This condition was first studied in \cite{BKM23}, where Bradford, Kohut, and the second author proved that the set $\mathbb{N}$ of natural numbers is not universal in $\mathscr{R}(\epsilon)$ for every $\epsilon \in (0,1]$. This result motivates the following question of Kolountzakis and Papageorgiou \cite{KolountzakisPapageorgiou23}, which can be viewed as an Erd\H{o}s similarity problem \textit{in the large}.

\begin{question}\label{q:large-erdos}
Does there exist an increasing sequence $a_n \to \infty$ that is universal in $\mathscr{R}(\epsilon)$ for some $\epsilon \in (0,1]$?
\end{question}

It is not clear if progress on either question implies progress on the other. Indeed, the assumption of large \textit{Lebesgue measure} distinguishes this question from the original Erd\H{o}s similarity problem, as well as other results from geometric measure theory where one seeks a set of zero Lebesgue measure but large Hausdorff dimension or Fourier dimension that does not contain any copy of a given configuration. See \cite{Keleti08, Mathe17, LiangPramanik} and the references therein.  

In \cite{KolountzakisPapageorgiou23}, Kolountzakis and Papageorgiou answered Question \ref{q:large-erdos} negatively for sub-exponential sequences, that is, sequences such that $\log{a_n} = o(n)$. 

\begin{externaltheorem}[Kolountzakis--Papageorgiou]\label{th:kolountzakispapageorgiou23}
    Let $\{a_n\}$ be a sequence of real numbers such that $a_0 = 0$, $a_{n+1} - a_n \geq 1$ for all $n$ in $\mathbb{N}$, and $\log{a_n} = o(n)$. For each $\epsilon \in (0,1]$, there exists a $(1-\epsilon)$-large set $S \subset \mathbb{R}$ that does not contain any affine copy of $\{a_n\}$.
\end{externaltheorem}

Their proof uses probabilistic techniques akin to those first developed by Kolountzakis in \cite{Kolountzakis97} to study the Erd\H{o}s similarity problem. 
Kolountzakis and Papageorgiou note that ``it remains an open question if a similar [set $S$] can be constructed when [the sequence $\{a_n\}$] grows exponentially or faster''.  Indeed, due to the nature of their probabilistic construction, some growth conditions must be imposed on $\{a_n\}$. In particular, it does not appear straightforward to modify their construction to handle a sequence $\{a_n\}$ with $\log a_n \neq o(n)$.


In Example \ref{ex:superexp}, we identify a family of sequences that can grow arbitrarily fast, and we construct explicit large sets $S$ that do not contain any affine copies of these sequences. This answers their question in the affirmative. We also provide new sufficient conditions for an increasing sequence $a_n \to \infty$ to be non-universal in $\mathscr{R}(\epsilon)$ for every $\epsilon \in (0,1]$. 

A key innovation in our approach is that for a given sequence $\{a_n\}$, we reduce Question \ref{q:large-erdos} to studying the size of the \textit{exceptional set} (of dilates)
\begin{align*}
    E:=E(\{a_n\}) = \{x \in \mathbb{R}: \langle xa_n\rangle \text{ is not dense in } [0,1)\},
\end{align*}
where $\langle y \rangle$ denotes the \textit{fractional part} of a real number $y$. We show that if $E$ is ``sufficiently small", then there exists a large set that does not contain any affine copy of $\{a_n\}$. To control the size of the set $E$, we employ new tools from metric number theory. In another direction, we show that it is sufficient to estimate the packing dimension of $E$. To the best of our knowledge, neither of these techniques have previously been applied to the Erd\H{o}s similarity problem or other similar problems.

In Section \ref{sec:sublacunary}, we prove the following theorem, which provides a novel connection between Question~\ref{q:large-erdos} and packing dimension.

\begin{thm}\label{thm: packing}
     Let $\{a_n\}$ be a sequence of real numbers. If the packing dimension of $E=E(\{a_n\})$ satisfies $\dim_P(E) < 1/2$, then for each $\epsilon \in (0,1]$, there exists a $(1-\epsilon)$-large set $S \subset \mathbb{R}$ that does not contain any affine copy of $\{a_n\}$.
\end{thm}

As a corollary of Theorem~\ref{thm: packing}, we deduce that sublacunary sequences are not universal in $\mathscr{R}(\epsilon)$ for every $\epsilon \in (0,1]$. This can be viewed as an analogue of a result on the Erd\H{o}s similarity problem, independently proved by Eigen \cite{Eigen85} and Falconer \cite{Falconer84}, which asserts that sublacunary sequences converging to zero are not universal in $\mathscr{E}$. 

\begin{cor}\label{cor:sublacunary-result}
Let $\{a_n\}$ be an unbounded sublacunary increasing sequence of real numbers. For each $\epsilon \in (0,1]$, there exists a $(1-\epsilon)$-large set $S \subset \mathbb{R}$ that does not contain any affine copy of $\{a_n\}$.
\end{cor}

Since every sublacunary sequence is sub-exponential, Corollary \ref{cor:sublacunary-result} has a smaller scope than Theorem \ref{th:kolountzakispapageorgiou23}. Nevertheless, our proof is simpler than the proof of Theorem \ref{th:kolountzakispapageorgiou23}, and, like the aforementioned result of Eigen and Falconer, it involves a deterministic construction. 

In Section~\ref{sec:proof-of-main-result}, we prove that sequences of real numbers whose integer parts form a set $A$ of positive \textit{upper Banach density}, that is, a set $A$ such that
\begin{align*}
    \limsup_{n \to \infty} \max_h \frac{\#(A \cap\{h + 1, h + 2, \ldots, h + n\})}{n} > 0,
\end{align*} 
are not universal in $\mathscr{R}(\epsilon)$ for every $\epsilon > 0$. 

\begin{thm}\label{th:main-result}
Let $\{a_n\}$ be an increasing sequence of real numbers such that the set $A = \{\lfloor a_n \rfloor\} \subset \Z$ has positive upper Banach density.
Then, for each $\epsilon \in (0,1]$, there exists a $(1-\epsilon)$-large set $S \subset \mathbb{R}$ that does not contain any affine copy of $\{a_n\}$.
\end{thm}

Notice that if a sequence $\{a_n\}$ is $1$-separated (that is, $a_{n+1}-a_n \geq 1$ for all $n \in \N$), then the upper Beurling density of $\{a_n\}$ agrees with the upper Banach density of $\{\lfloor a_n \rfloor\}$. We could therefore also have stated the above result in terms of upper Beurling density. Note also that since the above statement depends only on the set $A$ of integer parts of the sequence, and not on the sequence $\{a_n\}$, itself, the theorem is not sensitive to small perturbations of $\{a_n\}$. For example, if $\{b_n\}$ is any bounded sequence, then the conclusion holds for $\{a_n + b_n\}$. More generally, we may replace $\{b_n\}$ by any sequence that leaves the upper Banach density of the set $A$ unchanged.

Using Theorem \ref{th:main-result}, one can find examples of sequences of \textit{arbitrarily rapid growth} that are non-universal in $\mathscr{R}(\epsilon)$ for every $\epsilon \in (0,1]$. 

\begin{ex}\label{ex:superexp}
Let $\{a_n\}$ be the sequence obtained by arranging the elements of the set $\{f(i)+j: 1 \leq j \leq i\}$ in increasing order, where $f: \mathbb{N} \to \mathbb{R}$ is an increasing function. For any such sequence, $\{\lfloor a_n \rfloor\}$ has upper Banach density $1$. Therefore, Theorem \ref{th:main-result} implies that $\{a_n\}$ is not universal in $\mathscr{R}(\epsilon)$ for every $\epsilon \in (0,1]$. Furthermore, by choosing $f$ to grow sufficiently quickly, we can make $\{a_n\}$ grow as fast as we wish.
\end{ex}

We stress that since the sequences $\{a_n\}$ in Example~\ref{ex:superexp} can have super-exponential growth, they do not satisfy the hypotheses of Theorem \ref{th:kolountzakispapageorgiou23} (namely, $\log a_n = o(n)$), and therefore their non-universality was not previously known. However, notice that Theorem \ref{th:main-result} does not imply Theorem \ref{th:kolountzakispapageorgiou23} since there exist sub-exponential sequences of integers with zero Banach density.

Theorem \ref{th:main-result} allows us to handle sequences of arbitrarily rapid growth with deterministic constructions. The requirement that the integer parts of these sequences have positive upper Banach density forces them to contain arbitrarily long ``linear'' sections. While it is tempting to weaken this assumption by adapting the probabilistic construction used by Kolountzakis \cite[Theorem 3]{Kolountzakis97}, it is not clear whether the techniques used in \cite{Kolountzakis97, KolountzakisPapageorgiou23} extend to our large setting.


To prove Theorem \ref{th:main-result}, we give an explicit construction of a set $S$ with the required properties. The key step in our argument is the following theorem, which is of independent interest.

\begin{thm}\label{thm:Banach}
Let $\{a_n\}$ be a strictly increasing sequence of real numbers such that the set of integers $\{\lfloor a_n \rfloor\}$ has positive upper Banach density. Then, for any $\delta>0$ and any sequence of intervals $\{\Delta_n\}$  with length $\delta$, the set $$X=\{x \in \R: a_nx \not \in \Delta_n\!\!\pmod 1 \text{ for all } n \in \N \}$$ is countable.
\end{thm}

Theorem \ref{thm:Banach} generalizes results of Amice \cite{A64}, Kahane \cite{K64}, and Haight \cite{H88}. Amice and Kahane independently proved Theorem \ref{thm:Banach} assuming that $\{a_n\} \subset \Z$ and $\Delta_1 =\Delta_n$ for each $n$. With these stronger assumptions, they obtained the stronger conclusion that $X \cap [0,1]$ is finite. See also \cite{Kaufman68, Baker11} for higher-dimensional and quantitative versions of Amice and Kahane's result. Haight proved Theorem \ref{thm:Banach} under the stronger assumption that the set $\{\lfloor a_n \rfloor \}$ has positive upper asymptotic density. Our proof of Theorem \ref{thm:Banach} combines various ideas from analysis, number theory, and probability. We suspect that, with some extra work, our proof techniques could be extended to prove higher dimensional generalizations of Theorem~\ref{th:main-result} and Theorem~\ref{thm:Banach}.

In Section \ref{sec:exponential}, we investigate exponential sequences $\{b^n\}$, where $b>1$ is an algebraic number. We prove that such sequences are not universal in $\mathscr{R}(\epsilon)$ for all $\epsilon \in (1 - 1/\ell(b), 1]$, where $\ell(b)$ is the \textit{reduced length} of $b$ (see Section \ref{sec:exponential} for the definition). In the special case that $b \geq 2$ is an integer, we prove that $\{b^n\}$ is not universal in $\mathscr{R}(\epsilon)$ for every $\epsilon \in ((b-1)/b, 1]$.

\begin{thm}\label{thm:2n}
Let $b>1$ be an algebraic number. For each $\epsilon \in (1-1/\ell(b),1]$, there exists a $(1-\epsilon)$-large set $S \subset \mathbb{R}$ that does not contain any affine copy of $\{b^n\}$.
\end{thm}

Our proof of Theorem \ref{thm:2n} is inspired by results from Diophantine approximation by Dubickas \cite{D06, D06b} and uses certain special properties of the sequence $\{b^{n}\}$. We do not know if the range of $\epsilon$ can be expanded. 

Proving the non-universality of the sequence $\{2^{-n}\}$ in $\mathscr{E}$ has long been an impediment to resolving the Erd\H{o}s similarity problem. Similarly, it is not known whether the sequence $\{2^{n}\}$ is universal in $\mathscr{R(\epsilon)}$ for every $\epsilon \in (0,1]$: Theorem \ref{th:kolountzakispapageorgiou23} and Theorem \ref{th:main-result} do not apply to this sequence. In fact, even for a fixed $\epsilon$ close to $1$, neither theorem is applicable. The above theorem allows us to make progress towards this weaker question: the special case $b = 2$ yields that the sequence $\{2^n\}$ is not universal in $\mathscr{R}(\epsilon)$ for every $\epsilon \in (1/2, 1]$.

\begin{cor}
For each $\epsilon \in (1/2,1]$, there exists a $(1-\epsilon)$-large set $S \subset \mathbb{R}$ that does not contain any affine copy of $\{2^n\}$. 
\end{cor}

\textbf{Notation.} Throughout this paper, we denote by $\lfloor x \rfloor$ the largest integer that is no greater than the real number $x$. We denote by $\langle x \rangle$ the \textit{fractional part} of a real number $x$, that is, the unique real number in $[0,1)$ such that $\langle x \rangle = x - \lfloor x \rfloor$. As usual,  $\mathbb{R}$, $\mathbb{Q}$, $\mathbb{Z}$, $\mathbb{N}$ denote the real numbers, rational numbers, integers, and natural numbers respectively. For an algebraic number $b$, we write $\mathbb{Q}(b)$ for the smallest field containing both $\mathbb{Q}$ and $b$. For any subset $A$ of $\R$, we denote by $\#A$ the cardinality of $A$, and we write $|A|$ for its Lebesgue measure if $A$ is measurable.

\medskip

\textbf{Outline of the paper.} 
In Section \ref{sec:strategy}, we introduce a key lemma that provides a simple construction of a large set that avoids all affine copies of a given sequence $\{a_n\}$. In Section \ref{sec:sublacunary}, we prove Theorem~\ref{thm: packing}, which furnishes a connection between packing dimension and Question~\ref{q:large-erdos}. We also apply Theorem~\ref{thm: packing} to sublacunary sequences and prove Corollary~\ref{cor:sublacunary-result}. In Section \ref{sec:proof-of-main-result}, we prove 
Theorem~\ref{thm:Banach} and deduce Theorem~\ref{th:main-result}. Finally, in Section \ref{sec:exponential}, we construct large sets that avoid all affine copies of certain exponential sequences and prove Theorem~\ref{thm:2n}.

\section{A tool for constructing large sets}\label{sec:strategy}

Our main tool for constructing large sets is the following lemma, which generalizes \cite[Proposition 3.2]{BKM23}. It furnishes an explicit large set $S$ that does not contain any affine copy of $\{a_n\}$ in terms of the exceptional set $E(\{a_n\})$. 

For any sets $X, Y \subseteq \mathbb{R}$, we write $XY^{-1} = \left\{xy^{-1} : x \in X, y \in Y \setminus \{0\}\right\}$.  

\begin{lem}\label{lem: construction}
    Let $\{a_n\} \subset \mathbb{R}$ and $E = E(\{a_n\})$. If $EE^{-1} \neq \mathbb{R}$, then for each $\epsilon \in (0,1]$, there exists a $(1-\epsilon)$-large set $S \subset \mathbb{R}$ that does not contain any affine copy of $\{a_n\}$.
\end{lem}

\begin{proof}
    Fix $\epsilon \in (0,1]$ and $y \in \mathbb{R} \setminus (EE^{-1})$ with $y > 1$.  Such a choice is always possible since $E=-E$ and therefore $EE^{-1}$ is symmetric about 0.   
Choose $\alpha, \beta \in (0,1]$ so that $\beta \leq 1 - (1 - \epsilon + y)(1 + y)^{-1}$ and $2\alpha \leq \beta$. Let $T = \{x \in \mathbb{R} : \langle x \rangle \leq 1 - \alpha\}$ and $S = T \cap y T$. 
    
    Fix an arbitrary interval $I \subset \mathbb{R}$ of unit length and choose an integer $m$ so that $I \subset [m, m+2]$. Then 
        \begin{align*}
            \vert I \setminus T \vert &\leq \vert [m,m+2]\setminus T \vert = \vert  [m,m+1] \setminus T \vert + \vert [m+1, m+2] \setminus T \vert = 2\alpha \leq \beta,
        \end{align*}
    which implies that $\vert T \cap I \vert \geq 1 - \beta$. Next, fix arbitrary intervals of unit length $I \subset \mathbb{R}$ and $J \subset \mathbb{R}$ such that $y^{-1}I \subset J$. Then
    \begin{align*}
        \vert I \setminus S \vert &\leq \vert I \setminus T \vert + \vert I \setminus y T \vert = \vert I \setminus T \vert + y \vert (y^{-1} I) \setminus T \vert \leq \vert I \setminus T \vert + y \vert J \setminus T \vert \leq (1 + y) \beta \leq \epsilon,
    \end{align*}
    where the last inequality follows from our choice of $\beta$. This implies that $\vert S \cap I \vert \geq 1 - \epsilon$.

    Observe that $T$ cannot contain any affine copy of $\{a_n\}$ with dilation parameter in $\mathbb{R} \setminus E$; indeed, any such copy of $\{a_n\}$ will have terms with fractional part greater than $1 - \alpha$ since the sequence of fractional parts is dense in $[0,1)$.
    
    Since $S \subseteq T$, $S$ cannot contain any affine copy of $\{a_n\}$ with dilation parameter $y \in \mathbb{R} \setminus E$. Suppose for contradiction that there exist $t \in \mathbb{R}$ and $y' \in E \setminus \{0\}$ such that $t + y' \{a_n\} \subset S$. Then $t + y' \{a_n\} \subset y T$, and $(ty^{-1}) + (y'y^{-1})\{a_n\} \subset T$. This, together with the observation in the above paragraph, implies that $y'y^{-1} \in E$. But now $y'(y'y^{-1})^{-1} =  y$, so $y \in EE^{-1}$. This contradicts our choice of $y$.
\end{proof}

\begin{rem}
    Our construction of the large set $S$ only depends on the exceptional set $E(\{a_n\})$, and not on the sequence itself. For this reason, if a (possibly uncountable) family of sequences have the same exceptional set $E$, and $EE^{-1} \neq \R$, the construction in Lemma \ref{lem: construction} actually yields a \textit{single} large set that does not contain any affine copy of any sequence in this family. 
\end{rem}

\begin{rem}
Burgin, Goldberg, Keleti, MacMahon, and Wang consider a variant of Question \ref{q:large-erdos} where $\mathscr{R}(\epsilon)$ is replaced by $\mathscr{D}$, the collection of measurable sets $S \subseteq \mathbb{R}$ such that $\vert S \cap [m,m+1] \vert \to 1$ as $m \to \infty$. In \cite[Lemma 7]{BGKMW23}, they prove that if for every $j \in \mathbb{N}$ and every $\epsilon >0$ the tail $\{a_n: n \geq j\}$ of the sequence $\{a_n\}$ is not universal in $\mathscr{R}(\epsilon)$, then $\{a_n\}$ is not universal in $\mathscr{D}$. Combining this with Theorem \ref{th:kolountzakispapageorgiou23}, they deduce that sub-exponential sequences are not universal in $\mathscr{D}$. Since removing any finite number of terms from a sequence does not alter its exceptional set, we can apply their lemma to deduce that any sequence that satisfies the assumptions of Lemma \ref{lem: construction} is not universal in $\mathscr{D}$. The proofs of Theorem \ref{thm: packing} and Theorem \ref{thm:Banach} of the present paper make use of Lemma \ref{lem: construction} to establish the non-universality of sequences in $\mathscr{R(\epsilon)}$ for every $\epsilon \in (0,1]$. From the above observation, it follows that these sequences are also not universal in $\mathscr{D}$.
\end{rem}

In general, it is not easy to determine the exceptional set $E$ of a given sequence. In Section \ref{sec:sublacunary} and Section \ref{sec:proof-of-main-result}, we introduce new tools to give partial structural information on $E$.
Below, we list some well-known examples of sequences for which the corresponding exceptional set is known, and therefore Lemma~\ref{lem: construction} applies readily.

\begin{ex}
The exceptional set $E$ of each of the following sequences is countable. Therefore, $EE^{-1} \neq \R$, and Lemma~\ref{lem: construction} is applicable.
\begin{enumerate}[label=(\alph*)]
    \item \textit{Polynomial sequences.} Weyl's equidistribution theorem asserts that if $P$ is a nonconstant polynomial, then the sequence $\{P(n):n \in \N\}$ is uniformly distributed in $[0,1)$ if and only if at least one of the coefficients of $P$, other than the constant term, is irrational. Therefore, if $a_n = Q(n)$, where $Q$ is any nonconstant monic polynomial, we have that $\langle xa_n\rangle$ is not uniformly distributed in $[0,1)$ only if $x \in \Q$, and thus the exceptional set $E$ is contained in $\mathbb{Q}$. Notice that the special case $Q(x) = x$ implies \cite[Corollary 3.1]{BKM23}.
    \item \textit{Polynomial-like sequences supported on the primes.} Let $\{p_n\}$ be the set of prime numbers.  Bergelson, Kolesnik, Madritsch, Son, and Tichy \cite{BKMST} proved that if $\phi$ is a function of the form
    \begin{align}\label{eq:polynomial-like}
        \phi(x)=\sum_{j=1}^m \alpha_j x^{\theta_j},
    \end{align}
    where $0 <\theta_1 <\theta_2< \cdots< \theta_m$, $\alpha_j$ are nonzero real numbers (and if at least one $\alpha_j$ is irrational in the case that all $\theta_j\in \N$), then the sequence $\{\phi(p_n)\}$ is uniformly distributed. 

    Suppose that $a_n = \psi(p_n)$, where $\psi$ is a function of the form \eqref{eq:polynomial-like} with $\alpha_1 = 1$. Then we have that  $\langle xa_n\rangle$ is uniformly distributed in $[0,1)$ only if $x \in \Q$, and thus the exceptional set $E$ is contained in $\mathbb{Q}$.
    \item \textit{Sublacunary multiplicative semigroups.} Let $\{a_n\}$ be a sublacunary increasing sequence of positive integers which forms a \emph{multiplicative semigroup} 
 (closed under multiplication in $\mathbb{N}$). Furstenberg \cite{F67} proved that for any irrational number $x$, the sequence $\fract{x a_n}$ is dense in $[0,1)$. Thus, the exceptional set $E$ is contained in $\mathbb{Q}$. We refer the interested reader to generalizations of Furstenberg's theorem in \cite{K99, K18}.
    \item \textit{Certain sequences in a neighborhood of a sub-exponential sequence.} If $\{b_n\}$ is a sequence of sub-exponential growth and $\{d_n\}$ is any sequence with $d_n \to \infty$, Ajtai, Havas, and Koml\'os \cite{AHK83} showed that there exists a sequence $\{a_n\}$ with $b_n \leq a_n \leq b_n + d_n$ for which the sequence $\{xa_n\}$ is uniformly distributed modulo $1$ for every $x \in \R \setminus \{0\}$. In particular, the exceptional set of this sequence is simply $E = \{0\}$.
\end{enumerate}
\end{ex}


\section{Avoiding affine copies of sublacunary sequences}\label{sec:sublacunary}

In this section, we prove Theorem~\ref{thm: packing} and use it to deduce Corollary~\ref{cor:sublacunary-result}. We require the notions of \textit{Hausdorff dimension}, \textit{packing dimension}, and \textit{upper box dimension} from fractal geometry. 
We refer the reader to \cite{Falconer84} for the definitions. For any set $X \subseteq \mathbb{R}$, we denote these quantities by $\dim_H X, \dim_P X$, and $\overline{\dim_B} X$ respectively. They are related by the following inequalities (see \cite[Equation (3.29)]{Falconer84}):
\begin{align}\label{eq:hausdorff-packing-box}
    \dim_H X \leq \dim_P X \leq \overline{\dim_B} X.
\end{align}
In what follows, we also require three properties of Hausdorff dimension and packing dimension. The first is the inequality 
\begin{align}\label{eq:tricot}
    \dim_H (X \times Y) \leq \dim_H X + \dim_P Y,
\end{align}
proved in \cite[Theorem 3]{Tricot82}, and the second is the fact that the packing dimension is countably stable. This means that for any countable collection of subsets $A_n \subset \mathbb{R}$,
\begin{align*}
    \dim_P \left(\bigcup_{n = 1}^\infty A_n\right) = \sup_n \dim_P (A_n).
\end{align*}
The third property is the following fact, proved in \cite[Proposition 2.3]{Falconer84}. If $f : X \to \mathbb{R}$ is a Lipschitz function, then
\begin{align}\label{eq:lipschitz}
    \dim_H f(X) \leq \dim_H X. 
\end{align}

Theorem~\ref{thm: packing} uses packing dimension to provide a partial answer to Question~\ref{q:large-erdos}. We are now ready to prove it.

\begin{proof}[Proof of Theorem~\ref{thm: packing}]
    By Lemma \ref{lem: construction}, it suffices to show $EE^{-1} \neq \mathbb{R}$. Let $E_1 = E \cap (0, \infty) $ and $E_2 = E \cap (-\infty, 0]$. Observe that
    \begin{align*}
        E_1E_1^{-1} = \bigcup_{n=1}^\infty \bigcup_{m=1}^\infty \frac{E_1 \cap [1/n, m)}{E_1 \cap [1/n, m)}.
    \end{align*}

    Applying in turn the fact that the Hausdorff dimension is countably stable, that the map $f(x,y) = x/y$ is Lipschitz on $(E_1 \cap [1/n, m)) \times (E_1 \cap [1/n,m))$ for every $n,m \in \mathbb{N}$, and inequalities \eqref{eq:hausdorff-packing-box} and \eqref{eq:tricot}, we see that
    \begin{align*}
        \dim_H\left(E_1E_1^{-1}\right) &\leq \sup_{n,m} \dim_H\left(\frac{E_1 \cap [1/n, m)}{E_1 \cap [1/n, m)}\right) \\
         &\leq \sup_{n,m} \dim_H f\left[(E_1 \cap [1/n, m)) \times (E_1 \cap [1/n,m))\right] \\
        &\leq \sup_{n,m} \dim_H \left[(E_1 \cap [1/n, m)) \times (E_1 \cap [1/n,m))\right] \\
        &\leq \dim_H (E_1 \times E_1)\leq 2\dim_P(E_1).
    \end{align*}

    A similar argument shows that the same bound holds for $E_1E_2^{-1}$, $E_2E_1^{-1}$, and $E_2E_2^{-1}$. Also, since $\dim_P(E) < 1/2$, we have $\dim_H(E_iE_j^{-1}) < 1$ for $1 \leq i, j \leq 2$. It follows that $\dim_H(EE^{-1})<1$, and hence $EE^{-1} \neq \mathbb{R}$.
\end{proof}

With the help of Theorem~\ref{thm: packing}, Corollary \ref{cor:sublacunary-result} follows immediately from the next theorem, which is implicit in \cite{Boshernitzan94}.

\begin{externaltheorem}[Boshernitzan]
    Let $\{a_n\}$ be an unbounded increasing sequence of real numbers. If $\lim_{n \to \infty} a_{n+1}/a_n = 1$, then the set $E=E(\{a_n\})$ has zero packing dimension.
\end{externaltheorem}

\begin{proof}
The main result in \cite{Boshernitzan94} states that $E$ has zero Hausdorff dimension, but the proof of this theorem actually shows something stronger. Indeed, Boshernitzan shows that bounded granular sets (see \cite[Section 2]{Boshernitzan94} for a definition) have zero upper box dimension \cite[Proposition 3.3]{Boshernitzan94} and $E$ is a countable union of granular sets \cite[Proposition 4.1]{Boshernitzan94}. Since any granular set $G$ can be expressed as a countable union of bounded granular sets---namely, $G = \bigcup_{n \in \Z} (G \cap [n, n + 1])$, where the sets $G \cap [n,n+1]$ are granular since subsets of granular sets are granular---it follows from inequality~\eqref{eq:hausdorff-packing-box} that $E$ is a countable union of sets of zero packing dimension. Since the packing dimension is countably stable, we conclude that $E$ has zero packing dimension. 
\end{proof}

\begin{rem}
De Mathan \cite{Mathan80} and Pollington \cite{Pollington79} independently prove that if $\{a_n\}$ is a lacunary sequence, then its exceptional set has full Hausdorff dimension, and hence full packing dimension. Therefore, the above proof does not apply to any lacunary sequence. Note, however, that there exist sequences that are neither lacunary nor sublacunary, so the above result may apply to some of these sequences. In particular, the sequences in Example \ref{ex:superexp} are not sublacunary, but their exceptional sets have zero packing dimension.
\end{rem}


\section{Avoiding affine copies of sequences with positive upper Banach density}\label{sec:proof-of-main-result}
In this section, we prove Theorem \ref{th:main-result} and Theorem \ref{thm:Banach}. We begin by establishing two technical lemmas (Lemma \ref{lem: alphabeta} and Lemma \ref{lem: nonempty}) from which we deduce Theorem \ref{thm:Banach}. Next, we apply Theorem \ref{thm:Banach} to show that if $\{a_n\}$ is a sequence of real numbers whose integer parts form a set of positive upper Banach density, then the exceptional set $E(\{a_n\})$ is countable (Corollary \ref{cor:Banach}). Combining this observation with Lemma~\ref{lem: construction} immediately yields Theorem~\ref{th:main-result}.

\begin{lem}\label{lem: alphabeta}
Assume that $\alpha, \beta, \epsilon>0$ with $\alpha/\beta \not \in \Q$ and $\epsilon<\min(\alpha, \beta)$. For any two positive sequences of real numbers $\{y_m\}$ and $\{h_m\}$ with $y_m \to \infty$ as $m \to \infty$, we have that
\begin{align}\label{eq:lemma-41-eq}
\bigg|(\alpha\Z+[0,\epsilon]) \cap (\beta\Z+[0,\epsilon]) \cap [h_m, h_m+y_m]
\bigg|=\frac{\epsilon^2y_m}{\alpha \beta} +o(y_m) \qquad \text{as $m \to \infty$.}
\end{align}
\end{lem}

\begin{proof}
We may assume without loss of generality that $\alpha > \beta$. Let $\gamma=\alpha/\beta$. Since $\gamma \not \in \Q$, the sequence $\{\gamma n\}_{n=1}^{\infty}$ is uniformly distributed modulo $1$. Thus, by \cite[Chapter 2, Theorem 1.1]{KN74}, the discrepancy $D_N$ of the sequence $\{\gamma n\}_{n=1}^{\infty}$ tends to zero. This means that the quantity
$$
 D_{N}:=\sup_{I} \bigg|\frac{\#\{n \leq N: \gamma n \in I \pmod 1\}}{N}-|I| \bigg|,
$$
where the supremum is taken over all intervals $I$ of the torus, satisfies $D_{N} \to 0$ as $N \to \infty$.
In particular, for each interval $I$ of the torus and each real number $N>0$, we have
\begin{equation}\label{eq: IN}
\bigg|\#\{n \leq N: \gamma n \in I \pmod 1\}-|I|N\bigg|\leq ND_N.
\end{equation}    

For convenience, let us write
\begin{align*}
    A = \gamma\Z+\bigg[0,\frac{\epsilon}{\beta}\bigg], \quad B = \Z+\bigg[0,\frac{\epsilon}{\beta}\bigg], \quad \text{and} \quad C_m = \bigg[\frac{h_m}{\beta}, \frac{h_m+y_m}{\beta}\bigg].
\end{align*}
To prove equation \eqref{eq:lemma-41-eq} it suffices to show that
\begin{align*}
|A \cap B \cap C_m|=\frac{\epsilon^2y_m}{\alpha \beta^2} +o(y_m) \qquad \text{as $m \to \infty$}.
\end{align*}

We first rewrite the left-hand side of this equation as an integral:
\begin{align}
|A \cap B \cap C_m|
&=\int_{\R} \mathbf{1}_{A} (t) \mathbf{1}_{B}(t) \mathbf{1}_{C_m} (t) \,dt \notag \\
&=\sum_{k \in \Z} \int_{k}^{k+\epsilon/\beta} \mathbf{1}_{A} (t)  \mathbf{1}_{C_m} (t) \,dt \notag \\
&=\sum_{k \in \Z} \int_{0}^{\epsilon/\beta} \mathbf{1}_{A} (t+k) \mathbf{1}_{C_m} (t+k) \,dt \notag \\
&=\int_{0}^{\epsilon/\beta} \bigg(\sum_{k \in \Z}  \mathbf{1}_{A} (t+k)  \mathbf{1}_{C_m} (t+k)\bigg) \,dt \label{eq:tonelli} \\
&=\int_{0}^{\epsilon/\beta} \#\{x \in A \cap C_m: \fract{x}=t\} \,dt. \label{eq:integral-form}
\end{align}
Note that each $C_m$ is bounded and can therefore only contain finitely many terms $t + k$ for each $t$. Thus, the sum in equation \eqref{eq:tonelli} has finitely many nonzero terms, and the interchange of the integral and sum is justified.

Fix $t \in [0, \epsilon/\beta]$. We wish to estimate $\#\{x \in A \cap C_m: \fract{x}=t\}$. To this end, define
\begin{align*}
U_{m,t}=\bigg\{n \in \Z: \fract{\gamma n+z}=t, \frac{h_m}{\beta}\leq \gamma n+z \leq \frac{h_m+y_m}{\beta} \text{ for some } 0 \leq z \leq \frac{\epsilon}{\beta}\bigg\}.
\end{align*}
Since $\epsilon<\beta$, we have
\begin{align*}
    \# \{x \in A \cap C_m : \langle x \rangle = t\} &= \sum_{n \in \Z} \#\left\{x \in \left[\gamma n, \gamma n + \frac{\epsilon}{\beta}\right] \cap C_m : \langle x \rangle = t \right\}.
\end{align*}
Observe that
\begin{align*}
    \#\left\{x \in \left[\gamma n, \gamma n + \frac{\epsilon}{\beta}\right] \cap C_m : \langle x \rangle = t \right\} = \begin{cases}
        1 & \text{ if } n \in U_{m,t} \\
        0 & \text{ if } n \not\in U_{m,t}.
    \end{cases}
\end{align*}
Therefore,
\begin{align}\label{eq:discretize}
    \#\{x \in A \cap C_m: \fract{x}=t\}=\# U_{m,t}.
\end{align} 

To estimate $\# U_{m,t}$, we introduce a new set
\begin{align*}
V_{m,t}=\bigg\{n \in \Z: \fract{\gamma n+z}=t, \frac{h_m}{\beta}\leq \gamma n \leq \frac{h_m+y_m}{\beta} \text{ for some } 0 \leq z \leq \frac{\epsilon}{\beta}\bigg\}.
\end{align*}

We claim that $|\#U_{m,t}-\#V_{m,t}|\leq 1$. To see this, observe that if $n \in U_{m,t}\setminus V_{m,t}$ then $\gamma n < h_m / \beta \leq \gamma n + (\epsilon/\beta)$. Since $\gamma>1>\epsilon/\beta$ the intervals $[\gamma n, \gamma n + (\epsilon/\beta)]$ are disjoint for every $n$ and it follows that $\#(U_{m,t}\setminus V_{m,t})\leq 1$. A similar argument shows that $\#(V_{m,t} \setminus U_{m,t})\leq 1$. It follows that
$\vert \#U_{m,t} - \#V_{m,t} \vert=\vert \#(U_{m,t} \setminus V_{m,t}) - \#(V_{m,t} \setminus U_{m,t}) \vert\leq 1$.


This, together with the definition of $V_{m,t}$, shows that
\begin{align*}
\#U_{m,t}
=\#\bigg\{n \in \Z: \gamma n \in \bigg[t-\frac{\epsilon}{\beta},t\bigg] \pmod 1, \frac{h_m}{\beta}\leq \gamma n \leq \frac{h_m+y_m}{\beta}\bigg\}+O(1),
\end{align*}
where the implicit constant is independent of $m$ and $t$.

Let $n'=n'(m)$ be the smallest integer such that $\gamma n' \geq h_m/\beta$. Replacing $n$ with $n+n'$ in the above equation, we find that
\begin{align*}
\# U_{m,t} &=\#\bigg\{0 \leq n \leq \frac{y_m}{\beta \gamma}: \gamma n \in \bigg[t-\frac{\epsilon}{\beta}-\gamma n',t-\gamma n'\bigg] \pmod 1\bigg\}+O(1).
\end{align*}
Applying equation~\eqref{eq: IN} with $N_m=y_m/\beta \gamma$, we get
\begin{align}\label{eq:estimate-Umt}
\# U_{m,t} =\frac{\epsilon}{\beta} \cdot N_m +O(N_mD_{N_m}+1) = \frac{\epsilon}{\beta} \cdot N_m +o(N_m).
\end{align}
Note that the implicit constant in $o(N_m)$ is independent of $m$ and $t$.

We conclude that
\begin{align*}
\vert A \cap B \cap C_m\vert &= \int_{0}^{\epsilon/\beta} \#\{x \in A \cap C_m: \fract{x}=t\} \,dt && \quad \text{(by equation \eqref{eq:integral-form})}\\
&=\int_{0}^{\epsilon/\beta} \bigg(\frac{\epsilon}{\beta} \cdot N_m +o(N_m)\bigg) \,dt && \quad \text{(by equations \eqref{eq:discretize} and \eqref{eq:estimate-Umt})}\\
&=\frac{\epsilon^2y_m}{\alpha \beta^2}+o(y_m)
\end{align*}
as $m \to \infty$, as required.
\end{proof}

\begin{lem}\label{lem: nonempty}
Let $\{a_n\}$ be a strictly increasing sequence of real numbers such that the set $\{\lfloor a_n \rfloor\} \subset \Z$ has positive upper Banach density. Let $\eta \in (0,1)$ and define $A'=\bigcup_{n=1}^{\infty} [a_n,a_n+\eta].$
Let $\{\mu_n\}$ be a sequence of positive real numbers such that $\mu_{i}/\mu_{j} \not \in \Q$ for all $i \neq j$. If $\sum_{n=1}^{\infty} 1/\mu_n=\infty$, then for any $\epsilon>0$, the set
$A' \cap \bigcup_{n=1}^{\infty} (\mu_n \Z+[0,\epsilon])$
is nonempty.
\end{lem}
\begin{proof}
Fix $\epsilon>0$. Note that if there is an $n$ such that $\mu_n \leq \epsilon$, then $\mu_n \mathbb{Z}+[0,\epsilon]=\R$ and the statement of the proposition is trivial. Therefore, we may assume that $\epsilon<\mu_n$ for all $n$. Moreover, we may assume that $a_{n+1} - a_n > 2$. If this is not the case, apply the following argument to a subsequence of $\{a_n\}$ with this property.

Since the set $A=\{\lfloor a_n \rfloor\} \subset \Z$ has positive upper Banach density, we can find some $\delta'>0$ and two positive integer sequences $\{y_m\}$ and $\{h_m\}$ with $y_m \to \infty$ as $m \to \infty$, such that
\begin{align*}
    \#\big(A \cap \{h_m+1, h_m+2, \ldots, h_m+y_m\}\big)>\delta' {y_m}
\end{align*}
holds for all sufficiently large $m$. In particular, choosing $\delta = \delta'\eta/2$ we obtain the estimate
\begin{align}\label{eq:thickening}
    \big|A' \cap [h_m, h_m+y_m]\big|> \frac{\delta' \eta}{2} y_m = \delta y_m
\end{align}
for all sufficiently large $m$. 

For each $m$, let $\P_m$ be the uniform probability measure on the interval $[h_m / y_m, (h_m / y_m) + 1]$, and 
for each $i,m \in \N$, let 
\begin{align*}
B_{i,m}=\frac{(\mu_i\Z+[0,\epsilon]) \cap [h_m, h_m+y_m]}{y_m}.
\end{align*}
Observe that for each $i$, as $m \to \infty$, 
\begin{equation}\label{eq:i}
    \P_m(B_{i,m})=\frac{\epsilon}{\mu_i}+o(1).    
\end{equation}
For each $i \neq j$, since $\mu_{i}/\mu_{j} \not \in \Q$, Lemma~\ref{lem: alphabeta} implies that, as $m \to \infty$,
\begin{align}
\P_m(B_{i,m} \cap B_{j,m})
&=\bigg|\frac{(\mu_i\Z+[0,\epsilon]) \cap (\mu_j\Z+[0,\epsilon]) \cap [h_m, h_m+y_m]}{y_m}\bigg| \notag \\
&=\frac{\epsilon^2}{\mu_i\mu_j}+o(1) \notag \\
&=\P_m(B_{i,m})\P_m(B_{j,m})+o(1). \label{eq:ij}
\end{align}

Since $\sum_{n=1}^{\infty} 1/\mu_n=\infty$, we can choose $k$ such that
\begin{align}\label{eq:choice-of-k}
    \sum_{i=1}^k \frac{\epsilon}{\mu_i}>\frac{1}{\delta}+1.
\end{align}

We have from the Chung--Erd\H{o}s inequality \cite{CE52} that
\begin{align*}
\P_m \bigg(\bigcup_{i=1}^{k} B_{i,m}\bigg)
&\geq \frac {\left(\sum _{i=1}^{k}\P_m(B_{i,m})\right)^{2}}{\sum_{i=1}^{k}\sum _{j=1}^{k}\P_m(B_{i,m}\cap B_{j,m})}\\
&=\frac {\left(\sum _{i=1}^{k}\P_m(B_{i,m})\right)^{2}}{\sum_{\substack{1 \leq i,j \leq k \\i \neq j}}\P_m(B_{i,m}\cap B_{j,m})+\sum _{i=1}^{k}\P_m(B_{i,m})}.
\end{align*}

Then, as $m \to \infty$, 
\begin{align*}
\P_m \bigg(\bigcup_{i=1}^{k} B_{i,m}\bigg)
&\geq \frac{\left(\sum _{i=1}^{k}\P_m(B_{i,m})\right)^{2}}{\sum_{\substack{1 \leq i,j \leq k \\ i \neq j}}\big(\P_m(B_{i,m})\P_m(B_{j,m})+o(1)\big)+\sum_{i=1}^{k}\P_m(B_{i,m})} && \text{(by equation \eqref{eq:ij})} \\
&\geq \frac{\left(\sum _{i=1}^{k}\P_m(B_{i,m})\right)^{2}}{\left(\sum _{i=1}^{k}\P_m(B_{i,m})\right)^{2}+\left(\sum _{i=1}^{k}\P_m(B_{i,m})\right)+o(1)}\\
&= \frac{\left(\sum_{i=1}^k \frac{\epsilon}{\mu_i}+o(1)\right)^2}{\left(\sum_{i=1}^k \frac{\epsilon}{\mu_i}+o(1)\right)^2+ \sum_{i=1}^k \frac{\epsilon}{\mu_i}+o(1)} && \text{(by equation \eqref{eq:i})} \\
&> \frac{1}{1+\left({\frac{1}{\delta}+1+o(1)}\right)^{-1}} && \text{(by inequality \eqref{eq:choice-of-k})}.
\end{align*}
In particular, we can find a sufficiently large $m_0$  such that 
\begin{equation}\label{eq:union}
\bigg| \bigcup_{i=1}^{k} B_{i,m_0}\bigg|=\P_{m_0} \bigg(\bigcup_{i=1}^{k} B_{i,{m_0}}\bigg)\geq \frac{1}{1+\delta}>1-\delta.    
\end{equation}
Applying inequality~\eqref{eq:thickening} and inequality~\eqref{eq:union}, we find that
$$
\bigg| A' \cap [h_{m_0}, h_{m_0}+y_{m_0}]\bigg|+ \bigg| \bigcup_{i=1}^{k} (\mu_i \Z+[0,\epsilon]) \cap [h_{m_0}, h_{m_0}+y_{m_0}]\bigg|>y_{m_0}=\big|[h_{m_0}, h_{m_0}+y_{m_0}]\big|.
$$
We conclude that $A' \cap \bigcup_{i=1}^{k} (\mu_i \Z+[0,\epsilon]) \neq \emptyset. $
\end{proof}

We are now ready to prove Theorem~\ref{thm:Banach}.
\begin{proof}[Proof of Theorem~\ref{thm:Banach}]
We can assume that $\delta<1$. Without loss of generality, we can assume that $\Delta_n=[b_n, b_n+\delta]$ for each $n \in \N$, where $b_n \in [0,1)$. We show that $X \cap (0, \infty)$ is countable. The proof that $X \cap (-\infty,0)$ is countable is very similar. Thus, without loss of generality, we may assume that all elements in $X$ are positive. 

Suppose otherwise that $X$ is uncountable. Set $\rho=1+\delta/2$. Then we can find some integer $k$ such that $[\rho^k, \rho^{k+1}] \cap X$ is uncountable. In particular, we can find a sequence $\{\lambda_n\} \subset [\rho^k, \rho^{k+1}] \cap X$ such that $\lambda_{i}/\lambda_{j} \not \in \Q$ whenever $i \neq j$. Let $\mu_n=1/\lambda_n$ for each $n$. Note that $\sum_{n=1}^{\infty} 1/\mu_n=\infty$ and $\mu_{i}/\mu_{j} \not \in \Q$ whenever $i \neq j$. 

Since each $\lambda_i \in X$, we have that for every $n$, 
$a_n \lambda_i \not \in \Z+[b_n, b_n+\delta].$ Equivalently, $a_n \notin \mu_i\Z+[\mu_i b_n, \mu_i(b_n+\delta)]$. Note that $1/\rho^{k+1}\leq \mu_i \leq 1/\rho^k$, and
$$
\frac{b_n+\delta}{\rho^{k+1}}-\frac{b_n}{\rho^k}=\frac{b_n(1-\rho)+\delta}{\rho^{k+1}} \leq \frac{\delta+1-\rho}{\rho^{k+1}}=\frac{\delta}{2\rho^{k+1}}.
$$
Thus, $a_n \notin \mu_i\Z+[b_n/\rho^k,(b_n+\delta)/\rho^{k+1}]$.

Therefore, for each $i,n$, 
\begin{equation}\label{eq:empty}
a_n-b_n/\rho^k \notin (\mu_i\Z+[0,\delta/(2\rho^{k+1})]).    
\end{equation}
Let $\eta=\epsilon=\delta/(4\rho^{k+1})$. Let $a_n'=a_n- \eta -b_n/\rho^k$. Since $b_n/\rho^k \in [0,1/\rho^k)$, it is easy to verify that the set $\{\lfloor a_n' \rfloor\} \subset \Z$ still has positive upper Banach density. Let 
$A'=\bigcup_{n=1}^{\infty} [a_n',a_n'+\eta].$ Then equation~\eqref{eq:empty} implies that
$A' \cap \bigcup_{i=1}^{\infty} (\mu_i \Z+[0,\epsilon])=\emptyset,$ violating Lemma~\ref{lem: nonempty}.
\end{proof}

The special case of Theorem~\ref{thm:Banach} where $\Delta_1 = \Delta_n$ for every $n \in \N$ shows that the exceptional set associated with a given sequence in Theorem~\ref{th:main-result} is countable. 

\begin{cor}\label{cor:Banach}
Let $\{a_n\}$ be a strictly increasing sequence of real numbers such that the set of integers $\{\lfloor a_n \rfloor\}$ has positive upper Banach density.  Then $E=E(\{a_n\})$ is countable.
\end{cor}
\begin{proof}
For each $\alpha<\beta$, define $E_{\alpha,\beta} = \{x \in \mathbb{R} :
a_nx \not \in [\alpha, \beta] \pmod 1 \text{ for all } n \in \N\}$. Theorem~\ref{thm:Banach} implies that $E_{\alpha, \beta}$ is countable for each $\alpha<\beta$. We claim that $E \subseteq \bigcup_{\alpha<\beta, \alpha, \beta \in \Q}E_{\alpha, \beta}$. Indeed, if $x \in E$, then there exist $\theta \in \R$ and $\epsilon>0$ such that $a_nx \not \in (\theta-\epsilon, \theta+\epsilon) \pmod 1$ for all $n \in \N$, which implies that $x \in E_{\alpha, \beta}$ for some pair of rational numbers $\alpha<\beta$ in the interval $(\theta-\epsilon, \theta+\epsilon)$. Thus, $E$ is a subset of a countable set and must therefore be countable.
\end{proof}


\section{Avoiding affine copies of exponential sequences}\label{sec:exponential}

In this section, we use tools from Diophantine approximation to construct a large set that avoids all affine copies of $\{b^n\}$, where $b>1$ is an algebraic number. In particular, Theorem~\ref{thm:2n} follows immediately by taking the intersection of the set we construct in Corollary~\ref{cor:irrational} and the one we construct in Proposition~\ref{prop:countable} below by setting $B=\Q(b) \setminus \{0\}$.

The \textit{length} of a polynomial $F \in \R[x]$, denoted $L(F)$, is the sum of the absolute values of its coefficients. Consider the set of polynomials 
$$
\Gamma=\{a_0+a_1x+\cdots+a_mx^m \in \R[x]: m \in \N, a_0 = 1 \text{ or } a_m = 1\}.
$$
Given an algebraic number $b$, we can find its minimal polynomial $P(x) \in \Z[x]$, say 
$$
P(x)=a_dx^d+\cdots+a_1x+a_0,
$$
where $a_d>0$. The \emph{length} of $b$ is $L(b)=L(P)$ and the \emph{reduced length} of $b$, denoted $\ell(b)$, is
$$
\ell(b)=\inf_{G \in \Gamma} L(PG).
$$
Clearly, $\ell(b)\leq L(b)$. We refer to \cite[Section 6]{D06} and \cite{S06, S07} for a systematic study of the reduced length of algebraic numbers. If $b=p/q>1$, where $p,q$ are coprime positive integers, then $\ell(b)=\max\{p,q\}=p$ \cite[Corollary 1]{D06}.

Our construction below is motivated by the following theorem proved by Dubickas using tools from Diophantine approximation and combinatorics on words.

\begin{externaltheorem}[Dubickas]\label{thm:D06}
Let $b>1$ be an algebraic number. Then for each $\alpha \in \R \setminus \Q(b)$, the sequence $\{\fract{\alpha b^n}\}$ cannot be contained in an interval $I$ of the torus with length strictly smaller than $1/\ell(b)$.     
\end{externaltheorem}

\begin{proof}
Dubickas (\cite[Theorem 1]{D06}, \cite[Theorem 2]{D06b}) shows that if 
$\alpha$ is a nonzero number (and if $b$ is a Pisot number or a Salem number, then further assume that $\alpha \not \in \Q(b)$), then for any $t \in \R$, we have
$$
\limsup_{n \to \infty} \fract{\alpha b^n+t}-\liminf_{n \to \infty} \fract{\alpha b^n+t} \geq \frac{1}{\ell(b)}.
$$
Thus, whenever $\alpha \notin \Q(b)$, the sequence $\{\fract{\alpha b^n}\}$ cannot be contained in an interval $I$ of the torus with length strictly smaller than $1/\ell(b)$. 
\end{proof}

\begin{rem}
When $b=p/q$ for some coprime integers $p,q$ satisfying $1<q<p<q^2$,  Dubickas \cite{D09} proves that for each closed subinterval $I$ of length $1/p$ of the torus and each \emph{nonzero real} number $\alpha$, there are infinitely many $n \in \N$ such that $\fract{\alpha \cdot (p/q)^n}\notin I$. Note that $\ell(p/q)=p$. Thus, when $b=p/q$, the above theorem in fact holds whenever $\alpha \neq 0$.  We refer to the related discussion on the distribution of the fractional parts of powers of algebraic numbers in \cite[Chapter 3]{B12} and \cite{D19}.

On the other hand, very little is known about the distribution of the fractional parts of powers of transcendental numbers. We refer the interested reader to \cite[Chapter 2]{B12}.
\end{rem}

\begin{cor}\label{cor:irrational}
Let $\epsilon \in (0, 1/\ell(b))$ and let $b>1$ be an algebraic number. There exists a set $S \subset \R$ such that $\vert S \cap I \vert \geq 1/\ell(b) - \epsilon$ for every interval $I \subset \mathbb{R}$ of unit length, but $S$ does not contain any sequence of the form $\{\alpha b^n+t\}$, where $\alpha \in \R \setminus \Q(b)$ and $t \in \R$.
\end{cor}
\begin{proof}
Let
$$
S=\bigg\{x \in \R: 0\leq \fract{x}\leq \frac{1}{\ell(b)}-\epsilon\bigg\}.
$$
In view of Theorem~\ref{thm:D06}, $S$ cannot contain any affine copy of $\{b^n\}$ with a dilation factor outside $\Q(b)$.
\end{proof}

\begin{rem}\label{integerb}
Consider the case that $b \geq 2$ is an integer. Note that $\ell(b)=b$ and $\Q(b)=\Q$. Let $\epsilon \in (0, 1/b)$. Let $T=\{x \in \R: 0<\fract{x}<\frac{1}{b}-\epsilon/2\}$ be the set from Corollary \ref{cor:irrational}; then $T$ already already avoids all irrational copies of the sequence $\{b^n\}$, that is, all affine copies of the form $\{\alpha b^n+t\}$, where $\alpha \in \R \setminus \Q$ and $t \in \R$. Here we explicitly construct a set $S$ such that $\vert S \cap I \vert \geq 1/b - \epsilon$ for every interval $I \subset \mathbb{R}$ of unit length, that avoids \textit{all} affine copies of $\{b^n\}$.

A similar argument as in the proof of Lemma~\ref{lem: construction} shows that if $\beta$ is an irrational number, then $S=T \cap \beta T$ would do the job; unfortunately, note that the intersection between $S$ and a unit interval could be small. To resolve this issue, one can modify the construction described in \cite[Section 2]{BKM23} to avoid all rational copies of $\{b^n\}$ using the fact that $\{b^n\}$ is eventually periodic modulo any positive integer. The specific set resulting from this construction is $S = T \cap ((\R^+\setminus U) \cup - (\R^+\setminus U))$, where
\begin{align*}
   U= \bigcup_{k \in \N} \bigcup_{i=0}^{N-1} U_{k,i}, \quad  U_{k,i}=\bigcup_{m \in J_{k,i}}\bigg[m+\frac{i}{N}, m+\frac{i+1}{N}\bigg), \quad J_{k,i}=[2^{2^{kN+i}},2^{2^{kN+i+1}}] \cap \N,
\end{align*}
and $N$ is a positive integer such that $1/N<\epsilon/2$. 
\end{rem}

\medskip

In general, when $b$ is not necessarily an integer, we need to construct a large set that avoids those affine copies of $\{b^n\}$ with dilation factor in $\Q(b)$. However, the construction described in Remark~\ref{integerb} fails to extend to this more general setting due to its number theoretic nature. Fortunately, using the fact that $\Q(b)$ is a countable set, we can still obtain a suitable large set by following the more flexible construction below.

\begin{prop}\label{prop:countable}
    Let $\{a_n\}$ be an increasing sequence of real numbers such that $a_n \to \infty$. For every $\epsilon \in (0,1]$ and every countable set $B\subset \R \setminus \{0\}$, there exists a $(1-\epsilon)$-large set $S \subset \mathbb{R}$ that does not contain any subset of the form $b\{a_n\} + t$ with $b \in B$ and $t \in \R$.
\end{prop}
\begin{proof}
Let $\epsilon \in (0,1]$ and choose $N \in \N$ so that $\epsilon \geq 4/N$. We will construct four sets $T_1, T_2, T_3, T_4 \subset \R$, so that for each $i \in \{1,2,3,4\}$, $T_i$ satisfies $\vert T_i \cap [m,m+1]\vert \leq \frac{1}{N}$ for every $m \in \Z$ and $T_i$ contains an element in the sequence $b\{a_n\} + t$ for each $(b,t) \in K_i$,
where
\begin{align*}   
K_1=\{(b,t)\in B \times \R: b>0, t \geq 0\}, \\ 
K_2=\{(b,t)\in B \times \R: b>0, t<0\}, \\
K_3=\{(b,t)\in B \times \R: b<0, t \geq 0\}, \\
K_4=\{(b,t)\in B \times \R: b<0, t<0\}. 
\end{align*}
It then follows that $S=\R \setminus (T_1 \cup T_2 \cup T_3 \cup T_4)$ satisfies the two required properties that $\vert S \cap [m,m+1]\vert \geq 1-\frac{4}{N}\geq 1-\epsilon$ for every $m \in \Z$, and $S$ does not contain the sequence $b\{a_n\} + t$ for any $b \in B$ and $t \in \R$.

The constructions of the sets $T_1, T_2, T_3, T_4$ are similar. Below, we focus only on the construction of $T_1$. 

Let $B \cap \R^+=\{b_1, b_2, \ldots\}$. Let $D=\N \times (\N \cup \{0\})$. Define a map $f: D \to \N$ by
$$
f(m,n)=\frac{(m+n)(m+n-1)}{2} +m.
$$
For example, $f(1,0)=1, f(1,1)=2$, and $f(2,0)=3$. Note that $f$ is bijective; in fact, $f$ describes an enumeration of the countable set $D$. Thus, we can define $f^{-1}(k)=(m_k,n_k) \in D$ for each $k \in \N$.

Since $a_n \to \infty$, we may further assume without loss of generality that 
\begin{equation}\label{eq:grow1}
b_{m_k}a_{kN+i+1}> b_{m_k}a_{kN+i}+2 
\end{equation}
holds for each $k \in \N$ and $0 \leq i \leq N-2$, and 
\begin{equation}\label{eq:grow2}
b_{m_{k+1}}a_{(k+1)N}+n_{k+1}> b_{m_{k}}a_{kN+(N-1)}+n_{k}+2    
\end{equation}
holds for each $k \in \N$; otherwise, we simply pass to a subsequence of $\{a_n\}$. Define the following set 
\begin{align*}
T_1
&=\bigcup_{(m,n) \in D} \bigcup_{i=0}^{N-1}  \bigg[b_ma_{f(m,n)N+i}+n+\frac{i}{N}, b_ma_{f(m,n)N+i}+n+\frac{i+1}{N}\bigg]\\
&=\bigcup_{k \in \N}\bigcup_{i=0}^{N-1}  \bigg[b_{m_k}a_{kN+i}+n_k+\frac{i}{N}, b_{m_k}a_{kN+i}+n_k+\frac{i+1}{N}\bigg].    
\end{align*}
Note that the growth assumptions \eqref{eq:grow1} and \eqref{eq:grow2} guarantee that the intersection between $T_1$ and each unit interval has size at most $\frac{1}{N}$. Finally, we verify that $T_1$ contains an element in the sequence $b\{a_n\} + t$ for each $(b,t) \in K_1$. Indeed, if $b=b_m$, $\lfloor t \rfloor=n$, and $0 \leq i \leq N-1$ is the unique integer such that $\frac{i}{N} \leq \langle t \rangle <\frac{i+1}{N}$, then 
\begin{equation*}
    b_ma_{f(m,n)N+i}+t \in \bigg[b_ma_{f(m,n)N+i}+n+\frac{i}{N}, b_ma_{f(m,n)N+i}+n+\frac{i+1}{N}\bigg] \subset T_1. 
\end{equation*}
This completes the proof.
\end{proof}

\section*{Acknowledgments}
The authors thank Malabika Pramanik and Joshua Zahl for many helpful discussions and for their valuable feedback on a preliminary version of the manuscript. The authors also thank the anonymous referees for their valuable comments and suggestions. The first author is funded by a scholarship from China Scholarship Council, NSFC grant No.12171172, and Natural Science Foundation of Hubei Province No.2022CFB093. The second author was supported by an NSERC Discovery Grant GR010263. 

\bibliographystyle{amsalpha}
\bibliography{biblio.bib}

\end{document}